\documentclass[12pt]{amsart}

\usepackage{amsmath,amsthm,amssymb}
\usepackage{bbm}
\usepackage{mathrsfs}
\usepackage{enumerate}
\usepackage{graphicx}
\usepackage{caption} 
\usepackage{subcaption} 
\usepackage{thmtools}
\usepackage{thm-restate}
\usepackage{hyperref}

\newcommand{\del}{\partial }

\newtheorem{thm}{Theorem}
\newtheorem{conj}[thm]{Conjecture}
\newtheorem{prop}[thm]{Proposition}

\newtheorem{corr}[thm]{Corollary}
\newtheorem{lemma}[thm]{Lemma}
\theoremstyle{definition}

\newtheorem{defn}[thm]{Definition}
\newtheorem{question}[thm]{Question}
\theoremstyle{remark}

\begin{document}

\title{3--manifolds lying in trisected 4--manifolds}
\author{Dale Koenig}
\date{}

\begin{abstract}
The spine of a trisected 4--manifold is a singular 3-dimensional set from which the trisection itself can be reconstructed.  3--manifolds embedded in the trisected 4--manifold can often be isotoped to lie almost or entirely in the spine of the trisection.  We define this notion and show that in fact every 3--manifold can be embedded to lie almost in the spine of a minimal genus trisection of some connect sum of $S^2 \tilde \times S^2$s.  This mirrors the known fact that every 3--manifold can be smoothly embedded in a connect sum of $S^2 \times S^2$s.  Our methods additionally give an upper bound for how many copies of $S^2 \tilde \times S^2$ based on a distance calculated in an appropriately defined graph.  For the special case of lens spaces we analyze more closely and obtain more explicit bounds.
\end{abstract}

\maketitle

\section{introduction}

The theory of trisections of 4--manifolds was introduced by Gay and Kirby in 2012 \cite{GayKirby1}, and has since become a fast developing area of new research.  Trisections provide an analogue of Heegaard splittings in the theory of smooth 4--manifolds, allowing these 4--manifolds to be described and analyzed using only curves on a surface.  Meier and Zupan developed a theory of bridge trisections to analyze embedded surfaces in a trisected 4--manifold \cite{MeierZupan3,MeierZupan2}.  In this paper the author hopes to begin this program for embedded 3--manifolds. 

We present trisections of 4--manifolds in which certain 3--manifolds lie in a particularly nice way with respect to the trisection.    In these embeddings, the 3--manifold will lie almost entirely within the \emph{spine}, a singular 3--dimensional subset that completely determines the 4--manifold and its respective trisection.  We say that such 3--manifolds lie \emph{almost in the spine} of the 4--manifold.  These 3--manifolds can often be found directly from a nice enough diagram of the trisection.

Our main results center around 3--manifolds lying almost in the spine of minimal genus trisections of connect sums of $S^2 \times S^2$ and $S^2 \tilde \times S^2$.  It is well known that every closed 3--manifold embeds in some connect sum of $S^2 \times S^2$s.  We target these result from the point of few of trisections, requiring the 3--manifolds to lie almost in the spine of the 4--manifold.  We show that every lens space can be embedded to lie almost in the spine of a connect sum of only $S^2 \times S^2$s, and in the special case of lens spaces of the form $L(n,1)$ we show we can embed in a single copy of either $S^2 \times S^2$ or $S^2 \tilde \times S^2$.  More generally, we show that we can embed any orientable 3--manifold $M$ into $\#^n S^2 \tilde \times S^2$ where $n$ is related to the distance between the two cut systems determining a Heegaard splitting of $M$ in an appropriately defined graph.

These results naturally lead to connections with the more general theory of smooth embeddings of 3--manifolds in 4--manifolds. Aceto, Golla, and Larson defined the \emph{embedding number} $\varepsilon(M)$ of a 3--manifold $M$ to be the smallest value of $n$ for which $M$ embeds smoothly in $\#^n S^2 \times S^2$ \cite{AcetoGollaLarson1}.  The specifically focus on embeddings inheriting a spin structure from the 4--manifold.  Inspired by this, we define the \emph{spine embedding number} $\varepsilon_S(M)$ and \emph{twisted spine embedding number} $\tilde \varepsilon_S(M)$ of $M$ by restricting to embeddings in which $M$ lies almost in the spine of a trisection.  The results of this paper show $\tilde \varepsilon_S(M) < \infty$ for every closed orientable 3--manifold $M$, and give explicit upper bounds for both $\varepsilon_S(M)$ and $\tilde \varepsilon_S(M)$ when $M$ is a lens space.

\section{preliminaries and notation}

All 4--manifolds will be taken to be smooth, closed, and orientable unless specified otherwise.   The 3--manifolds are closed, orientable, and piecewise smooth.  We will use $X$ to denote a four dimensional manifold and $M$ to denote a three dimensional manifold which we will embed inside $X$.  A \emph{trisection} of $X$ is a decomposition $X = X_1 \cup X_2 \cup X_3$ such that each $X_i$ is a four dimensional handlebody and each pairwise intersection $Y_{ij} = X_i \cap X_j$ is a three dimensional handlebody \cite{GayKirby1}.  The triple intersection $\Sigma = X_1 \cap X_2 \cap X_3$ is a closed surface which is the boundary of each of the three 3 dimensional handlebodies $Y_{ij}$.  The union $Y = Y_{12} \cup Y_{13} \cup Y_{23}$ is called the \emph{spine} of the trisection.  If $\Sigma$ is a genus $g$ surface and the $X_i$ are genus $k_i$ handlebodies then we say that the trisection is a $(g;k_1,k_2,k_3)$ -- trisection.

A \emph{cut system} on the surface $\Sigma$ is a set of $g$ curves $C_1 ,\cdots, C_g$ such that $\Sigma - \left( \bigcup_i C_i \right)$ is a sphere with $2g$ punctures.  A cut system determines a three dimensional handlebody with $\Sigma$ as its boundary, obtained by attaching disks to $C_i \times \{0\} \subset \Sigma \times [0,1]$ for each $i$, and then filling in the resulting $S^2$ boundary component with a 3--ball.  The spine of the trisection can be reconstructed from the surface $\Sigma$ and three marked cut systems on $\Sigma$, the three cut systems determining how $\Sigma$ bounds the three handlebodies $Y_{ij}$.  We will sometimes refer to the three $Y_{ij}$ as $Y_\alpha, Y_\beta$, and $Y_\gamma$,  and we will refer to the cut systems determining them as the $\alpha$ curves, $\beta$ curves, and $\gamma$ curves respectively.  Thus $Y_{\alpha}$ is determined by a set of curves $\alpha_1,\cdots,\alpha_g$ on $\Sigma$, $Y_\beta$ by $\beta_1,\cdots,\beta_g$, and $Y_\gamma$ by $\gamma_1,\cdots,\gamma_g$.  Note that the cut systems are restricted in that any pair of cut systems must give a Heegaard diagram for a connect sum of $S^1 \times S^2$s. 

A genus $g$ Heegaard diagram for $\#^k S^1 \times S^2$ given by curves $\alpha_1,\cdots,\alpha_g$ and $\beta_1,\cdots,\beta_g$ is called \emph{standard} if, after possibly reindexing, $\alpha_i$ and $\beta_i$ are parallel for $1 \le i \le  k$ and $\#|\alpha_i \cap \beta_j| = \delta_{ij}$ otherwise.  This extends to $k=0$ so $M=S^3$, in which $\#|\alpha_i \cap \beta_j| = \delta_{ij}$ for all $i,j$.  A Heegaard splitting of such a manifold is called \emph{standard} if it can be given by such a diagram.  Note that a standard Heegaard splitting will have many nonstandard diagrams, but every diagram can be made standard by handle slides and isotopy.  Heegaard splittings $\#^k S^1 \times S^2$s are all standard, so there is a unique Heegaard splitting of genus $g$ for all $g \ge k$ \cite{waldhausen1}.  

In this paper we will investigate embeddings of the 3--manifold $M$ into $X$ such that $M$ lies ``nicely" with respect to the trisection of $X$.  We will henceforth assume that $X$ comes equipped with a trisection.  To simplify the proceeding discussion we will first make a definition.

\begin{defn}
A \emph{subhandlebody} of a three dimensional handlebody is a connected component resulting from making compressions along some set of disjoint nonparallel essential disks. 
\end{defn}

Since the disk compressions are chosen to be nonparallel, a subhandlebody can only be a ball if the original handlebody was.  If $H$ is a subhandlebody of $Y_{ij}$, then $\del H$ consists of the union of some connected subsurface of $\Sigma = \del Y_{ij}$ with a collection of disjoint nonparallel essential disks $D_1, \cdots, D_m$ in $Y_{ij}$.  $H$ is topologically still a handlebody, and we can diagrammatically view $H$ by marking in $\Sigma$ the curves $\del D_i$ and an additional nonseparating collection of $g(H)$ curves bounding disks in $H$.

\begin{defn}
Suppose $M$ is a 3--manifold with $S^2$ boundary components.  $M$ lies \emph{in the spine} of the trisection of $X$ if $M \subset Y$ and each $M \cap Y_{ij}$ is a union of subhandlebodies.
\end{defn}

\begin{defn}
Suppose $M$ is a closed orientable 3--manifold.  $M$ lies \emph{almost in the spine} if each $M \cap Y_{ij}$ is a union of handlebodies and each $M \cap \mathring X_i$ is a union of 3-balls.
\end{defn}

We will often abuse terminology and say that $M$ lies (almost) in the spine of $X$ to mean that $M$ lies (almost) in the spine of the equipped trisection of $X$.  Note that since the spine is not a smooth object, 3--manifolds in or almost in the spine will not be embedded smoothly.  However, we can give the 4--manifold a PL structure compatible with its smooth structure in which the spine and embedded 3--manifolds are PL.

Given a 3--manifold $M$ with $S^2$ boundary components lying in the spine of $X$, we can often extend the embedding to that of a closed 3--manifold lying almost in the spine of $X$.  Let $M$ denote such a 3--manifold.  For each $S^2$ boundary component $A \subset \del M$, $A - A \cap \Sigma$ consists of essential disks in the $Y_{ij}$.  We can color these disks by which $Y_{ij}$ they lie in.
\begin{prop}
If $A$ is such a boundary component of $M$  and if at most two colors appear on $A$ (that is, if some color does not appear) then $A$ bounds a 3-ball lying in one of the $X_i$.  If every such $S^2$ boundary component bounds a 3-ball in this way, then attaching these 3-balls extends the embedding of $M$ to an embedding of $\tilde M$ lying almost in the spine of $X$.
\end{prop}
\begin{proof}
This follows from Laudenbach and Po{\'e}naru's proof that a diffeomorphism on the boundary of a 4-dimensional handlebody extends to the interior \cite{LaudenbachPoenaru1}. Since at most two colors appear on $A$, $A$ lies in the boundary of some $X_i$.  $A$ then bounds a unique 3-ball in the interior of that $X_i$ up to isotopy, since we can apply a diffeomorphism of $\#^k S^1 \times S^2$ to take it to a 2-sphere that ``obviously" bounds a 3--ball, and this diffeomorphism extends to the interior of the handlebody.  \end{proof}

When there are multiple 2-sphere boundary components, they can be chosen to bound disjoint 3--balls.  Thus, if no $S^2$ boundary component has disks of all three colors on it, then the embedding of $M$ can be extended to an embedding of a closed manifold almost in the spine of $X$.

\section{Diagrams and simple cases}

We will henceforth assume that any boundary components that can be filled are filled, so there are punctures iff some $S^2$ boundary components of $M \cap Y$ have disks of all three colors.  In this case, $M$ is uniquely determined by the collection of subhandlebodies $M \cap Y_{ij}$.  We can then define a diagram for the pair $(X,M)$.

\begin{defn}
A  \emph{diagram for the pair $(X,M)$} is a trisection diagram of $X$ together with an additional diagram on $\Sigma$ in which the curves determining each subhandlebody component of $M \cap Y_{ij}$ are drawn in the appropriate color.
\end{defn}

We can draw all the subhandlebody diagrams on the same copy of $\Sigma$ since they are in different colors.  However we do need a separate copy of $\Sigma$ on which to specify the normal trisection diagram.  We let $M_\alpha = M \cap Y_\alpha$, $M_\beta = M \cap Y_\beta$, and $M_\gamma = M \cap Y_\gamma$ be the collections of subhandlebodies in the three $Y_{ij}$, so a diagram for $(X,M)$ consists of a normal trisection diagram together with another copy of $\Sigma$ on which $\alpha$ curves are given specifying $M_\alpha$, $\beta$ curves specifying $M_\beta$, and $\gamma$ curves specifying $M_\gamma$.

We now investigate when 3--manifolds can lie almost in the spine of a trisection for low genus trisections.  Note since the union of two of the 3-dimensional handlebodies in the spine gives a connect sum of $S^1 \times S^2$s by definition, we can often find uninteresting sums of $S^1 \times S^2$ lying in the spine.  We therefore only consider the case where $M$ is not such a connect sum.

\begin{prop}
If $M = L(n,1)$ then $M$ can be embedded almost in the spine of a genus $2$ trisection of $S^2 \times S^2$ or $S^2 \tilde \times S^2$ - the former if $n$ is even and the latter if $n$ is odd.
\label{embedding_n_1}
\end{prop}

\begin{proof}
Let $X$ be the trisection diagram shown in Figure \ref{simplestcase}, where the number of times the left blue curve loops around clockwise is determined by the integer parameter $n$.  We claim that this trisection demonstrates the desired properties.  This requires us first to show that $L(n,1)$ embeds almost in the spine, and second to show that we have indeed given a diagram for $S^2 \times S^2$ (if $n$ is even) or $S^2 \tilde \times S^2$ (if $n$ is odd).
\begin{figure}[h]
\centering
\includegraphics[height=1.5in]{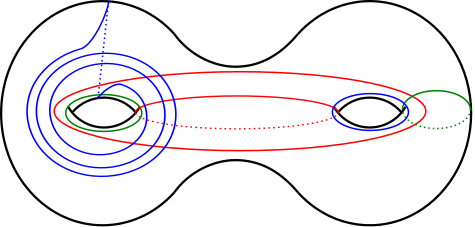}
\caption{A trisection diagram in which $L(n,1)$ embeds almost in the spine, where $n$ is the number of times the left blue curve loops around longitudinally.}
\label{simplestcase}
\end{figure}f

For the first statement, we find $L(n,1)$ broken up as shown in Figure \ref{simplestcasepieces}.  To see that this is indeed $L(n,1)$, note that if we treat the red curves as determining one genus 2 handlebody, and treat the left blue curve and the right green curve as determing a second genus 2 handlebody, then we have a genus 2 Heegaard diagram for $L(n,1)$.  Indeed, the union of the blue solid torus, the green solid torus, and the grey ball in Figure \ref{simplestcasepieces} gives the second genus 2 handlebody of the splitting.

\begin{figure}[h]
\centering
    \begin{subfigure}[b]{0.8\textwidth}
        \centering
        \includegraphics[height=1.5in]{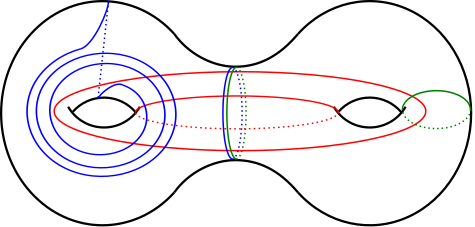}
        \caption{Curves determining $M_\alpha,M_\beta$ and $M_\gamma$.}
    \end{subfigure}%

    \begin{subfigure}[b]{0.8\textwidth}
        \centering
        \includegraphics[height=2in]{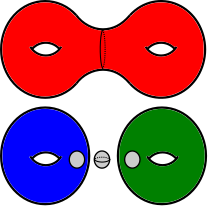}
        \caption{ $M_\alpha, M_\beta, M_\gamma$, and the ball lying outside the spine.}
    \end{subfigure}
\caption{The curves drawn in (A) determine three handlebodies as shown in (B).  The union of these three handlebodies is a punctured copy of $L(n,1)$ which can be filled in by a ball lying outside the spine.}
\label{simplestcasepieces}
\end{figure}
Now we identify $X$ from the trisection diagram.  In fact, this can be done simply by checking that we have indeed given a valid balanced trisection diagram, calculating the intersection form, and applying the classification of genus $2$ trisections by Meier and Zupan \cite{MeierZupan1}.  However, we will show it directly using Kirby calculus.  Think of the two blue curves as lying in the copy of $S^3$ determined by the Heegaard splitting into the green and red handlebodies.  We think of the surface as two punctured tori glued along the central curve, with the red handlebody being the copy of $\{\mathrm{punctured~torus}\} \times I$ connecting them.  The right blue curve then lies on the inner punctured torus as a longitude, and the outer $(n,1)$ curve lies as an $(n,1)$ curve on the outer punctured torus.  The framings on the blue curves are then $0$ and $n$ respectively, coming from the surface.  We then forget the surface, resulting in the Kirby diagram shown in Figure \ref{simplestcasekirby}.  This is a Hopf link with framings $0$ and $n$, and this diagram is known to correspond to $S^2 \times S^2$ when $n$ is even and $S^2 \tilde \times S^2$ when $n$ is odd (see page 144 of \cite{GompfStipsicz1}).  Note that one of the two curves will be reflected from how the blue curve looked on the original surface.  In the figure we have chosen to reflect the $(n,1)$ curve, as this closer matches the construction used later.  The choice of curve to reflect corresponds to an orientation reversing diffeomorphism of the copy of $S^3$ containing the red and green handlebody.

\begin{figure}[h]
\centering
\includegraphics[height=2.5in]{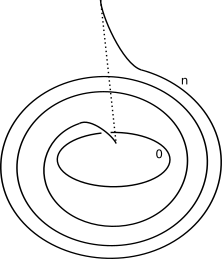}
\caption{A Kirby diagram coming from the trisection in which $L(n,1)$ embeds almost in the spine.}
\label{simplestcasekirby}
\end{figure}
\end{proof}

We conjecture that this construction is the only interesting possibility when $g=2$. 
\begin{conj}
If a 3--manifold $M$ that is not $\#^k S^1 \times S^2$ can be embedded to almost lie in the spine of a genus 2 trisection of a 4--manifold $X$, then $M$ is a $L(n,1)$ lens space and $X$ is $S^2 \times S^2$ or $S^2 \tilde \times S^2$.  Moreover, $X$ is the former if $n$ is even and the latter if $n$ is odd.
\end{conj}

We will conclude this section by constructing an example of a punctured $M$ that lies in the spine of a trisection but cannot be extended to an embedding of a closed manifold lying almost in the trisection.  Consider the trisection shown in Figure \ref{puncturedthinga}.  This is a spun $L(3,2)$ as constructed in \cite{meier1}.  The diffeomorphism type is that of $\mathcal{S}_3$, one of Pao's manifolds \cite{Pao1}.  The diagram in Figure \ref{puncturedthingb} shows a once punctured copy of $L(3,2)$ lying in the spine.  The $S^2$ boundary has a disk of each color when colored by the intersections with the $Y_{ij}$, so it cannot be filled in by a ball lying in some $X_i$.
\begin{figure}[h]
\centering
    \begin{subfigure}[b]{0.8\textwidth}
        \centering
        \includegraphics[height=1.8in]{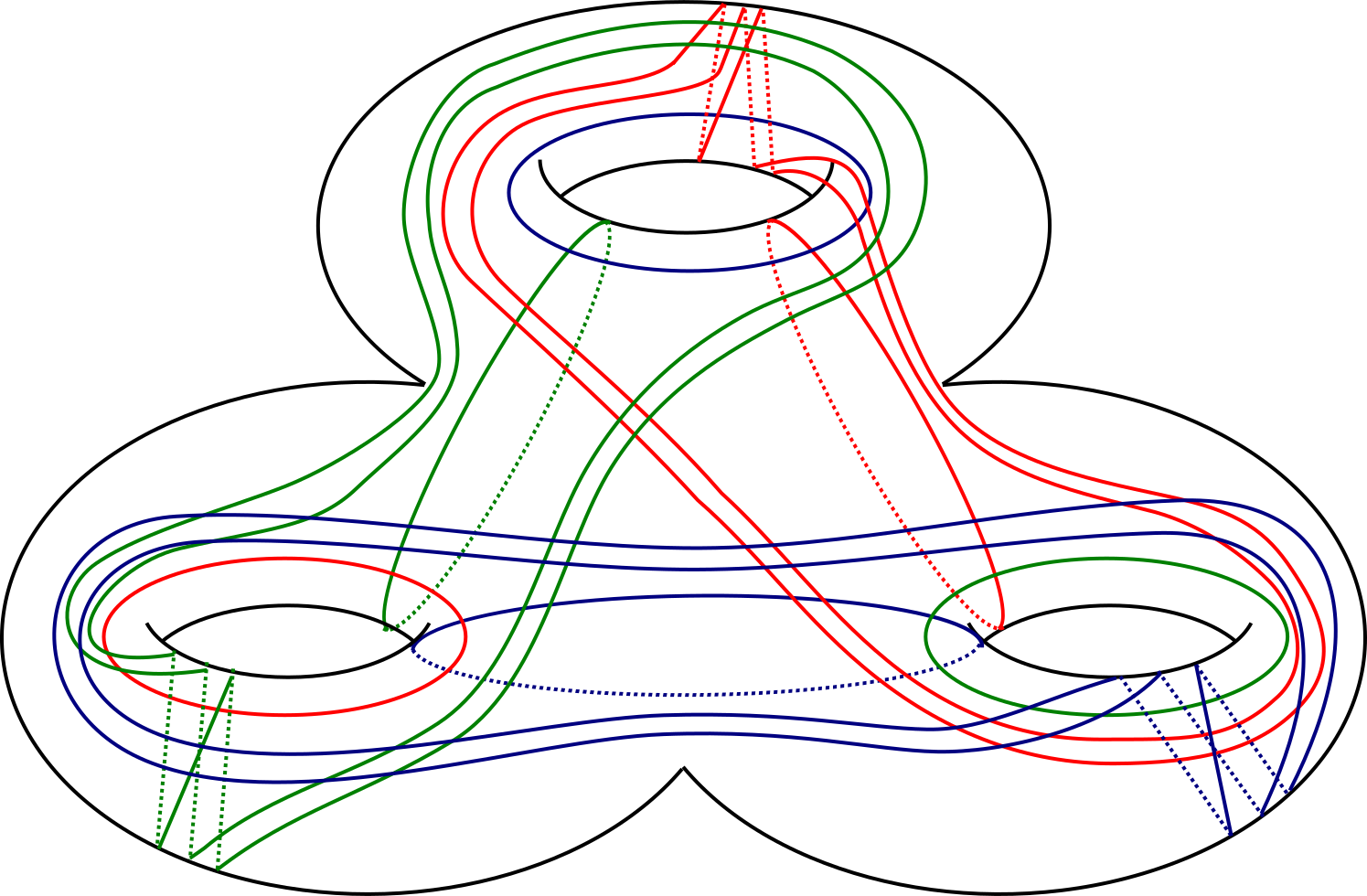}
        \caption{A trisection of a spun lens space.}
        \label{puncturedthinga}
    \end{subfigure}%

    \begin{subfigure}[b]{0.8\textwidth}
        \centering
        \includegraphics[height=1.8in]{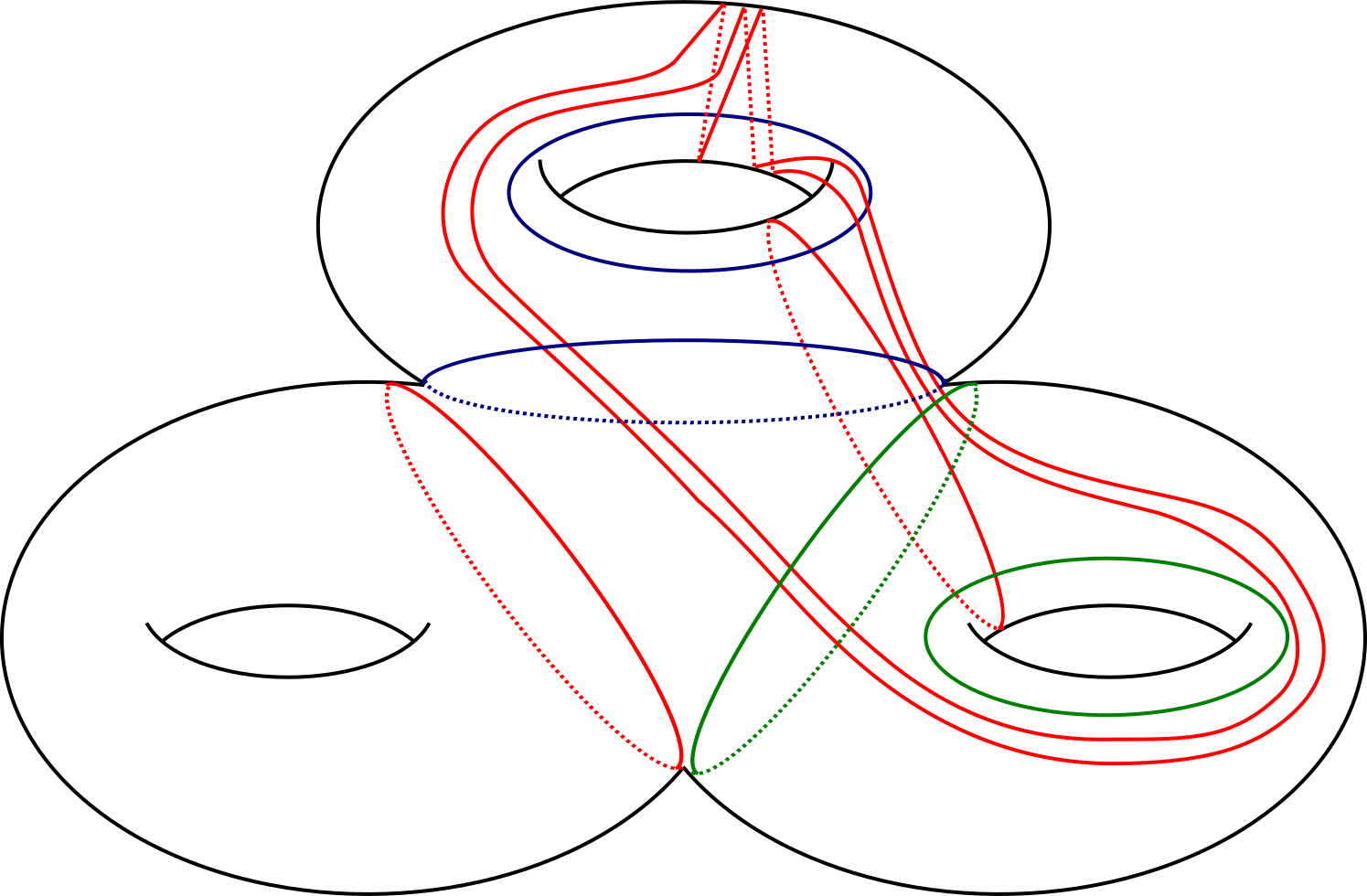}
        \caption{ Three subhandlebodies whose union is $L(3,2)$. }
                \label{puncturedthingb}
    \end{subfigure}
    
        \begin{subfigure}[b]{0.8\textwidth}
        \centering
        \includegraphics[height=1.6in]{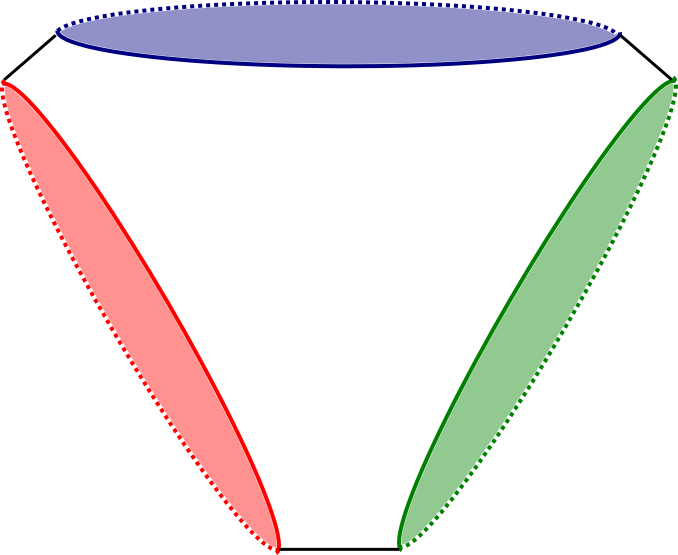}
        \caption{ The $S^2$ boundary component.}
    \end{subfigure}
\caption{A punctured copy of $L(3,2)$ lying in the spine of a spun $L(3,2)$.}
\label{puncturedthing}
\end{figure}

\section{The graphs $\mathcal{A}$ and $\mathcal{B}$}

We must now define two graphs $\mathcal{A}$ and $\mathcal{B}$ derived from a surface $S$.  For both graphs, the set of vertices is the set of cut systems for $S$ up to isotopy.  In $\mathcal{A}$, two vertices are connected by an edge if the two cut systems are dual.  That is, if the vertex $v_1$ represents a cut system $a_1,\cdots,a_g$ and the vertex $v_2$ represents a cut system $b_1,\cdots,b_g$ then $v_1$ and $v_2$ are connected by an edge iff we have that $\#|a_i \cap b_j| = \delta_{ij}$ for all $1 \le i,j \le g$, up to isotopy and reindexing.  Put more simply, two vertices are connected by an edge if the two cut systems give a standard Heegaard diagram for $S^3$.  In $\mathcal{B}$, we connect two distinct vertices by an edge if they represent cut systems such that each pair of curves is either dual or parallel.  In other words, we two vertices are connected by an edge if the two cut systems give a standard Heegaard diagram for $\#^k S^1 \times S^2$ for any $k$.  We first show that these graphs are connected, so that we can define metrics $d_\mathcal{A}$ and $d_\mathcal{B}$, on the set of cut systems of $S$.  

\begin{prop}
The graph $\mathcal{B}$ is connected.
\end{prop}
\begin{proof}
Hatcher and Thurston defined a graph $X_g$ in which vertices correspond to cut systems on a genus $g$ surface, and edges correspond to pairs of cut systems in which one pair of curves is dual and the rest are parallel (so the two cut systems give a standard Heegaard diagram for $S^1 \times S^2$).  They additionally proved that $X_g$ is connected \cite{HatcherThurston1}.  Our graph $\mathcal{B}$ has the same set of vertices and contains every edge of $X_g$, so it must also be connected.
\end{proof}

\begin{prop}
The graph $\mathcal{A}$ is connected.
\end{prop}
\begin{proof}

We again use the fact that the graph $X_g$ defined by Hatcher and Thurston is connected.  We replace each edge of $X_g$ with a sequence of two edges in $\mathcal{A}$ connecting the same two vertices as shown in Figure \ref{Bpathfull}.  Note that for any two cut systems connected in by an edge in $X_g$ there is a diffeomorphism of the surface taking them to the cut systems in the top left and top right of the figure.
\begin{figure}[h]
\centering
\includegraphics[height=2in]{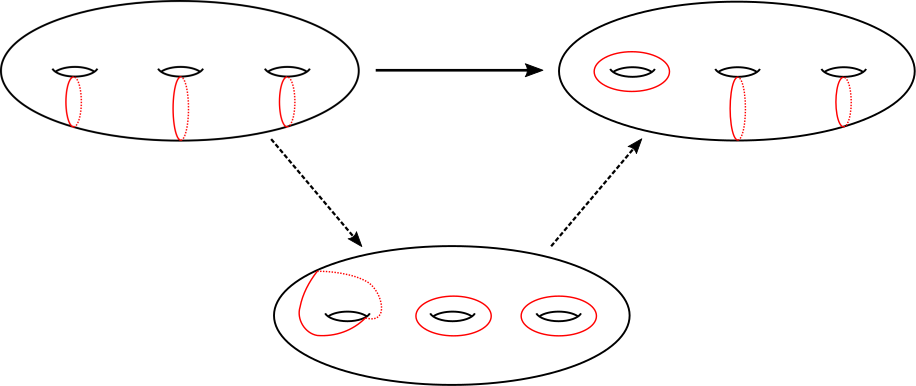}
\caption{We replace the top edge (an edge in $X_g$) with two edges in $\mathcal{A}$, passing through an intermediate cut system. }
\label{Bpathfull}
\end{figure}
\end{proof}

We note that $\mathcal{A}$ is a subgraph of $\mathcal{B}$, so $d_\mathcal{B}(v_1,v_2) \le d_\mathcal{A}(v_1,v_2)$ for any two cut systems $v_1,v_2$.

We can also take the set of vertices to be the set of cut systems up to both slides and isotopy instead of only isotopy.  This gives new graphs $ \mathcal{A}'$ and $\mathcal{B}'$ with associated metrics $d_{\mathcal{A}'}$ and $d_{\mathcal{B}'}$.  Since there exist natural quotient maps $\mathcal{A} \to \mathcal{A'}$ and  $\mathcal{B} \to \mathcal{B'}$, we see that $\mathcal{A'}$ and $\mathcal{B'}$ are also connected, and that $d_{\mathcal{A}'} \le d_\mathcal{A}$ and $d_{\mathcal{B}'} \le d_\mathcal{B}$.  

The graphs $\mathcal{A'}$ and $\mathcal{B'}$ can be used to define an invariant of Heegaard splittings as follows:  If $S$ is a Heegaard surface for a 3--manifold, define $d_{\mathcal{A'}}(S)$ to be the distance in $\mathcal{A'}$ between cut systems defining the two sides of the splitting, and define $d_{\mathcal{B'}}(S)$ similarly.  This distance is well defined by the way we defined the vertex sets for $\mathcal{A'}$ and $\mathcal{B'}$.  Both distances are $\ge 2$ unless $M$ is a connect sum of $S^1 \times S^2$s.

\section{Embedding more complicated 3--manifolds}

\begin{thm}
Assume $M$ is a 3--manifold that is not a connect sum of $S^1 \times S^2$s, and $S$ is a genus $g$ Heegaard surface.  Then $M$ lies almost in the spine of a minimal genus trisection of $\#^n S^2 \times S^2$ or $\#^n S^2 \tilde \times S^2$ where $n = g(d_{\mathcal{A}'}(S)-1)$.
\label{generalembeddingthm}
\end{thm}

\begin{proof}
We view the trisection surface $\Sigma$ as $2(d_{\mathcal{A}'}(S)-1)$ copies of $S$ laid side by side and with the adjacent copies connected by a tube.  Between each adjacent pair, draw $2g$ symmetric nonparallel curves, alternating between green and red for each adjacent pair of copies of $S$.  On the rightmost copy of $S$, draw $g$ meridians in green, and on the leftmost copy draw $g$ longitudes in red.  At this point the diagram should look like Figure \ref{basicpattern}, with the pattern extended for higher $g$ or $d_{\mathcal{A}'}(S)$.

Now since $\mathcal{A}'$ is connected we know that there is a sequence of diagrams $\mathcal{D}_0,\cdots,\mathcal{D}_m$ where $m=d_{\mathcal{A}'}(S)$, the pair $\mathcal{D}_0$ and $\mathcal{D}_m$ determine a Heegaard diagram for $M$, and each adjacent pair $\mathcal{D}_i,\mathcal{D}_{i+1}$ gives a (possibly nonstandard) diagram for $S^3$.  By applying slides on the diagrams and a diffeomorphism to $S$ we may assume that $\mathcal{D}_0$ is a set of standard meridians and $\mathcal{D}_1$ a set of standard longitudes dual to those meridians.  On the copies of $S$, from left to right, we then draw $\mathcal{D}_2, \mathcal{D}_3, \cdots, \mathcal{D}_{m-1}, \mathcal{D}_{m},\mathcal{D}_{m-1}, \cdots, \mathcal{D}_2, \mathcal{D}_1$ in blue to give the third set of curves.  When $i$ is even we draw $\mathcal{D}_i$ reflected across the vertical axis in the plane.  This will result in the desired trisection diagram.  If $M$ is a lens space and $S$ a genus 1 Heegaard surface, then $d_{\mathcal{A'}}(S) = 2$ and the trisection produced is precisely that constructed in Proposition \ref{embedding_n_1}.

\begin{figure}[h]
\centering
\includegraphics[height=2.5in]{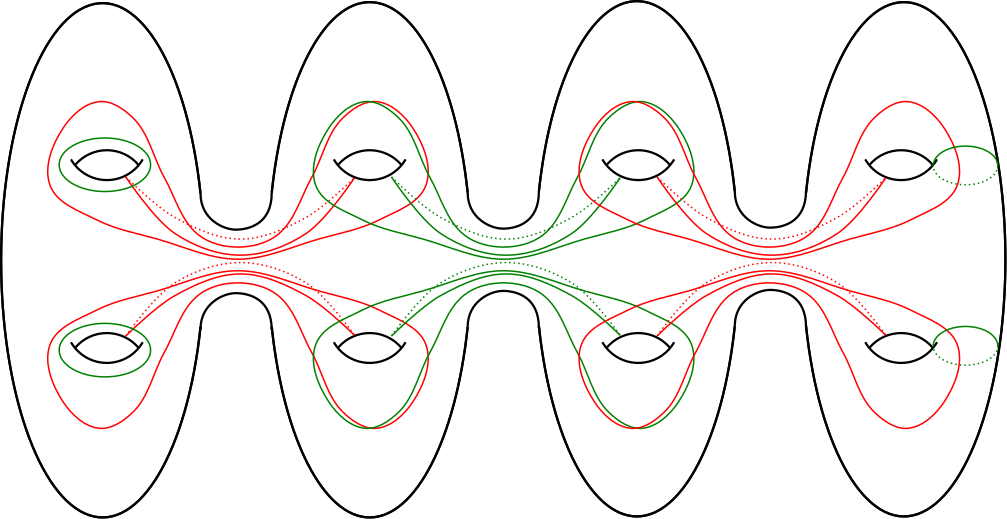}
\caption{The set of red and green curves when $g=2$ and $d_{\mathcal{A}'}(S) = 3$.  The surface is drawn as $2(3-1)=4$ copies of a genus $2$ surface tubed together sequentially.}
\label{basicpattern}
\end{figure}

Let us first look more closely at a slightly more complicated example.  Let $M = L(7,2)$.  Then by setting $\mathcal{D}_1$ to be a $(1,0)$ curve, $\mathcal{D}_2$ a $(3,1)$ curve and $\mathcal{D}_3$ a $(7,2)$ curve we see that $d_{\mathcal{A}'}(S) = 3$.  Our construction produces the trisection diagram shown in Figure \ref{harderlensspace}.  We find a copy of $L(7,2)$ lying in the trisection using the curves shown in Figure \ref{harderlensspacefindingM}.  The curves determine a genus 3 green handlebody, a genus 2 red handlebody, and a genus 1 blue handlebody, all lying in the spine.  The union of these three handlebodies together with two balls lying outside the spine is a copy of $L(7,2)$.  To see this, we take our decomposition into three subhandlebodies and two 3-balls, collapse the three balls, and then glue the genus 1 and 2 subhandlebody together to form a genus 3 handlebody.   Diagrammatically we obtain a Heegaard diagram given by first compressing the leftmost shared green and blue curve to get a genus 3 surface, and then recoloring the blue and red curves to the same color to determine compressing disks for one side of the splitting.  The green curves determine the other side of the splitting.  This diagram can be shown to be a twice stabilized diagram for $L(7,2)$ as shown in Figure \ref{harderlensspaceheeg}.

\begin{figure}[h]
\centering
    \begin{subfigure}[b]{\textwidth}
        \centering
        \includegraphics[height=1in]{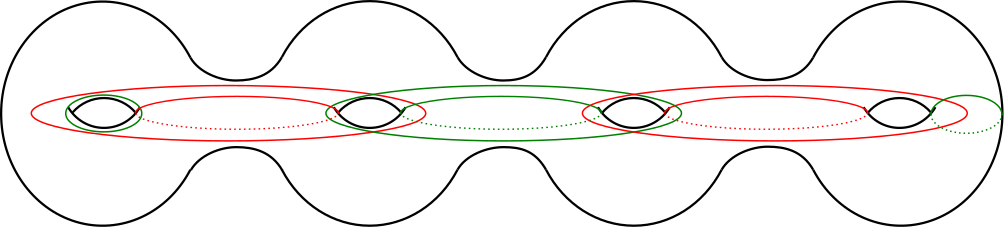}
        \caption{The red and green curves for the diagram.  This part will always be as shown when $g=1$ and $d_{\mathcal{A}'}(S) = 3$.}
    \end{subfigure}%

    \begin{subfigure}[b]{\textwidth}
        \centering
        \includegraphics[height=1in]{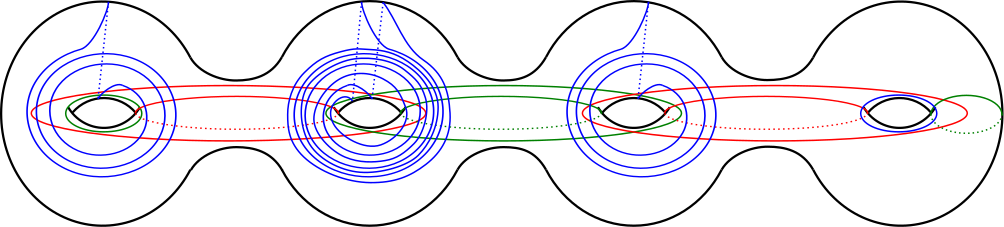}
        \caption{ The full trisection diagram.  Note that the $(3,1)$ curves have been reflected to be $(-3,1)$ curves.}
    \end{subfigure}
\caption{Constructing a trisection diagram for $\#^2 S^2 \tilde \times S^2$ such that $L(7,2)$ lies almost in the spine.}
\label{harderlensspace}
\end{figure}

\begin{figure}[h]
\centering
\includegraphics[height=1in]{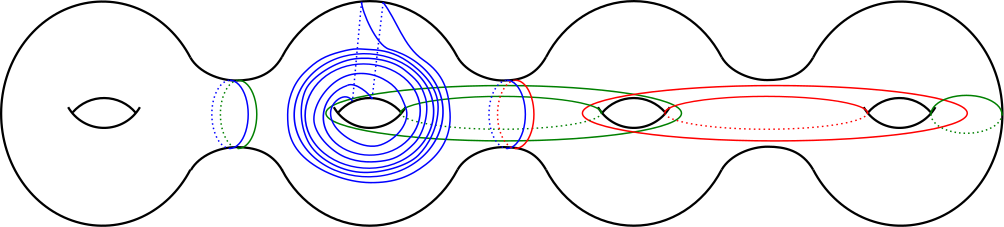}
\caption{A copy of $L(7,2)$ in the trisection.}
\label{harderlensspacefindingM}
\end{figure}

\begin{figure}[h]
\centering
    \begin{subfigure}[b]{\textwidth}
        \centering
        \includegraphics[height=1in]{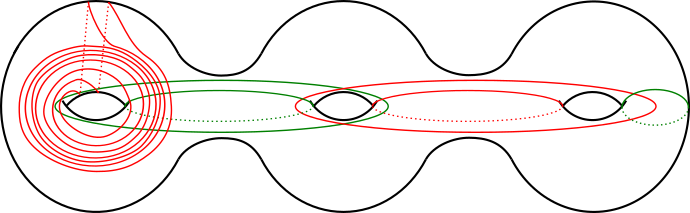}
        \caption{The diagram after restricting to the appropriate genus 3 surface and recoloring the blue curve to red.}
    \end{subfigure}%

    \begin{subfigure}[b]{\textwidth}
        \centering
        \includegraphics[height=1in]{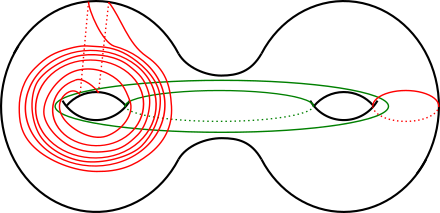}
        \caption{ The diagram after one destabilization.}
    \end{subfigure}
        \begin{subfigure}[b]{\textwidth}
        \centering
        \includegraphics[height=1in]{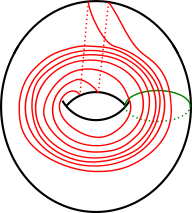}
        \caption{ The diagram after two destabilizations.}
    \end{subfigure}
\caption{The Heegaard diagram for $L(7,2)$.  The right most green meridian is dual to a red curve, compressing along the green curve and removing the dual red curve corresponds to a destabilization of the splitting.  After this destabilization there is a similar destabilization visible with the red meridian.  After both destabilizations we are left with standard diagram of $L(7,2)$.}
\label{harderlensspaceheeg}
\end{figure}

We now identify $X$ from the trisection diagram.  Note that the red and green curves determine a Heegaard splitting of $S^3$.  We view the surface as lying in $S^3$ as depicted in Figure \ref{ins3}, folded over itself repeatedly to look like several parallel copies of a the torus $S$ connected by tubes.  The blue curves lie on this as the copies of $\mathcal{D}_1, \mathcal{D}_2,\mathcal{D}_3,\mathcal{D}_2$ from the inner torus to the outer torus.  So the (1,0) curve $\mathcal{D}_1$ lies on the innermost torus, a (3,1) curve on the second torus, the $(7,2)$ curve on the third torus, and another $(3,1)$ curve on the outermost torus.  Note that the earlier reflections of the (3,1) curves were undone by the process of folding the surface.  The four blue curves give a Kirby diagram for $X$, with the framings of the curves given by the surface.  For a (p,q) curve on a torus, the framing induced by the torus is precisely $pq$, so the framings of the four blue curves are $0,3,14$, and $3$ respectively from the inside outwards.  Moreover, curves with adjacent indices have linking number $\pm 1$ by construction.

If the inner (3,1) curve is slid over the outer (3,1) curve using an outwardly pointing radial arc to guide the slide, the result is a 0 framed unknot that links the (7,2) curve once.  This 0 framed unknot can then be used to unknot the (7,2) curve as well as unlink it from any other components.  These two curves can be pushed off two the side as a Hopf link with framings 0 and 14, which is a diagram for $S^2 \times S^2$.  We are left with the (1,0) curve and one of the (3,1) curves.  The (1,0) curve is a 0 framed unknot linking the (3,1) curve once.  Again have a Hopf link, this time with framings 0 and 3, giving a diagram for $S^2 \tilde \times S^2$.  Putting the two pieces together shows that the 4--manifold is $S^2 \times S^2 \# S^2 \tilde \times S^2 \cong \#^2 S^2 \tilde \times S^2$.  
\begin{figure}[h]
\centering
\includegraphics[height=2in]{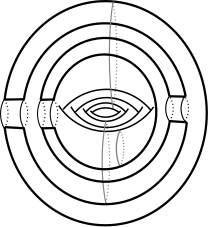}
\caption{The genus 4 trisection surface lying in the copy of $S^3$ given by the green and red curves.  The surface appears as four parallel tori connected by tubes.}
\label{ins3}
\end{figure}

We return to the general case.  The Kirby diagram is given by $\mathcal{D}_1,\cdots,\mathcal{D}_m,\mathcal{D}_{m-1},\cdots,\mathcal{D}_2$ lying on parallel copies of $S$, with their framings given by the embedding in $S$.  $\mathcal{D}_i$ and $\mathcal{D}_{i+1}$ are dual sets of curves after possibly performing some slides on both diagrams.  Then sliding curves from the inner $\mathcal{D}_{m-1}$ over the same curves from the outer $\mathcal{D}_{m-1}$ gives a set of 0 framed unknots that each link one of the curves of $\mathcal{D}_m$ once.  They can therefore be used to unknot the curves of $\mathcal{D}_m$ as well as unlink them from the rest of the curves, resulting in several Hopf links split from the rest of the diagram.  Each Hopf link has one zero framed component, so it represents a copy of $S^2 \times S^2$ or $S^2 \tilde \times S^2$ depending on whether the framing of the other component is even or odd.  After pulling away these pairs, the rest of the diagram is given by copies of $\mathcal{D}_1,\cdots,\mathcal{D}_{m-1},\mathcal{D}_{m-2},\cdots,\mathcal{D}_2$ (now peaking at $m-1$) lying on parallel copies of $S$, and the process can be repeated by sliding curves of the inner copy of $\mathcal{D}_{m-2}$ over curves of the outer copy.  The final diagram is that for a connect sum of $S^2 \times S^2$s and $S^2 \tilde \times S^2$s.
\end{proof}

Note that in the above proof we have shown that the trisection is minimal genus and identified the manifold, but we have not identified the trisection itself.  That is, the trisection may or may not be standard.  We also observe that a careful analysis of how the embedded 3--manifold in this construction lies in the 4--manifold shows that the 3--manifold will always be separating.

Standard 4--manifold theory tells us that $(S^2 \times S^2 )\# (S^2 \tilde \times S^2)$ is diffeomorphic to $ \#^2 S^2 \tilde \times S^2$.  Therefore we can connect sum the 4--manifold obtained in the theorem with the standard genus 2 trisection of $S^2 \tilde \times S^2$ to get the following simpler but weaker statement.
\begin{corr}
Every orientable 3--manifold embeds almost in the spine of a minimal genus trisection of some $\#^n S^2 \tilde \times S^2$.
 \end{corr}

Similarly, by stabilizing enough times we can assume the trisection is standard at the cost of losing control over the genus.

It is well known that any 3--manifold embeds in some $\#^n S^2 \times S^2$ \cite{GompfStipsicz1}.  It is natural to ask if this is also true when we require the 3--manifolds to lie almost in the spine of a trisection.  For lens spaces, at least, this is the case.
\begin{prop}
Every lens space $L(p,q)$ lies almost in the spine of a minimal genus trisection of some $\#^n  S^2 \times S^2$.
\end{prop}
\begin{proof}
First we must observe that copies of $S^2 \times S^2$ pop up when pairs of curves separate out as Hopf links with framings 0 and $n$ with $n$ even, and $S^2 \tilde \times S^2$ when pairs of curves separate out as Hopf links with framings 0 and $n$ with $n$ odd.  The framing is determined by the Heegaard surface into which the curve is embedded. If the Heegaard surface $S$ is genus 1, then every curve is a $(p,q)$ curve, and the framing coming from the surface is $pq$.  Thus, we will get copies of $S^2 \times S^2$ whenever $pq$ is even, and copies of $S^2 \tilde \times S^2$ when both $p$ and $q$ are odd.  This leads us to search for a sequence of diagrams $\mathcal{D}_i$, where $\mathcal{D}_i$ is a $(p_i,q_i)$ curve, $(p_0,q_0) = (0,1)$, $(p_1,q_1) = (1,0)$, $(p_m,q_m) = (p,q)$ and all products $p_i q_i$ are even.

For a genus 1 surface, the graphs $\mathcal{A}$ and $\mathcal{A'}$ are isomorphic to the Farey graph.  The Farey graph is the graph with vertices labelled by reduced fractions $0 \le p/q \le 1$ and edges between pairs $p_1/q_1$ and $p_2/q_2$ where $p_1q_2 - p_2 q_1 = \pm 1$.   Therefore there is an edge iff a $(p_1,q_1)$ curve and a $(p_2,q_2)$ curve intersect geometrically once.  We define the \emph{even Farey graph} to be the subgraph with vertices the set of reduced fractions $p/q$ where $pq$ is even.  We use $(p,q)$ and $p/q$ interchangeably to represent these vertex labels.
\begin{lemma}
The even Farey graph is connected.
\end{lemma}
\begin{proof}
The Farey graph can be drawn as a planar graph, where the faces are triangles corresponding to triples $(p_1,q_1),(p_2,q_2),(p_3,q_3)$ with each $p_i q_j-q_i p_j = \pm 1$.  Taking this equation mod 2 shows that one of $p_i,q_i,p_j,q_j$ must be even.  Therefore if some pair $p_i,q_i$ are both odd, then both other pairs must lie in the even Farey graph.

A finite iteration of the Farey graph is shown in Figure \ref{fareydiagram}.  The diagram is planar, with each face a triangle.  The diagram is constructed iteratively starting from the points $1/0$ and $0/1$.  At each iteration, add a vertex at the midpoint of the circle arc connecting two already present vertices, and connect it to those vertices to form a new triangle.  The sizes of the triangles decrease with each successive iteration,  and when some point $p/q$ is added to the diagram it must be added as the outermost point of a triangle, the other two points coming from previous iterations.  We combine this with our previous observation that exactly two of the three points of each triangle must lie in the even Farey graph to see that every vertex lying in the even Farey graph (other than $0/1$ and $1/0$) is connected to exactly one vertex of the even Farey graph coming from a previous iteration of the normal Farey graph.  Tracing this path back eventually gives a path in the even Farey graph connecting $p/q$ with either $1/0$ or $0/1$.  It follows that the even Farey graph is indeed connected.  We can also see that if $p/q$ lies between $1/0$ and $1/1$ in the diagram, the path will reach $1/0$ before $0/1$, since it cannot travel through $1/1$.
\begin{figure}[h]
\centering
\includegraphics[height=4in]{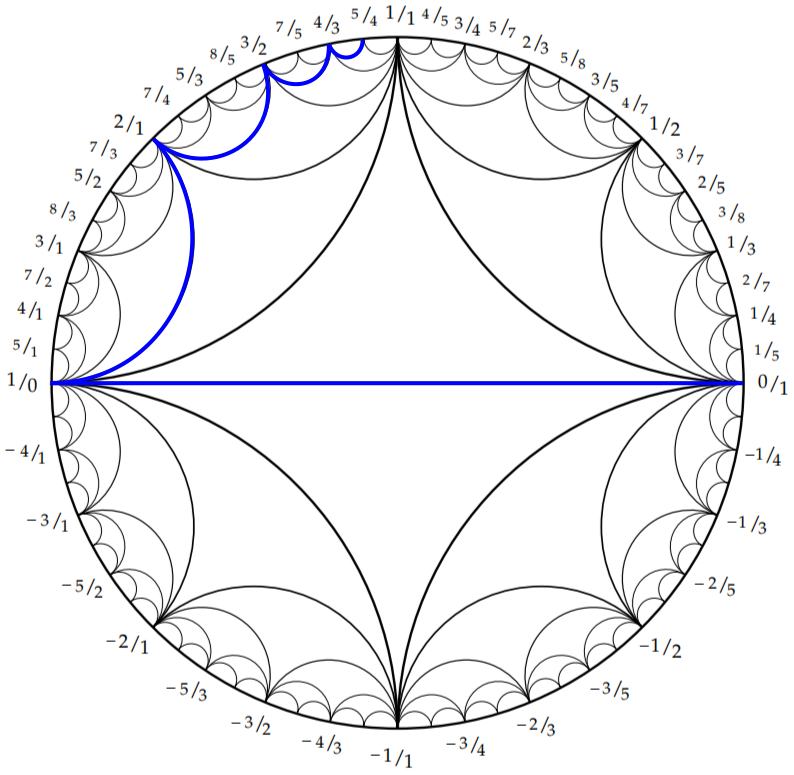}
\caption{A section of the Farey graph, with the blue path lying in the even Farey graph and connecting $0/1$ with $5/4$.  Adapted with permission from \emph{Topology of Numbers} (p. 16) by Allen Hatcher.  Retrieved from \url{https://www.math.cornell.edu/~hatcher/TN/TNbook.pdf} in January 2018.}
\label{fareydiagram}
\end{figure}
\end{proof}

Now if both $p$ and $q$ are odd we can replace $L(p,q)$ with the homeomorphic $L(p,p-q)$.  Since $(p,p-q)$ lies in the even Farey graph and the even Farey graph is connected, there is a sequence of $(p_i,q_i)$ connecting $(1,0)$ to $(p,p-q)$ where $p_i q_i$ is always even.  Using this sequence of diagrams ensures that every Hopf link that gets split off is a Kirby diagram for $S^2 \times S^2$, so the final 4--manifold is a connect sum of $S^2 \times S^2$s.
\end{proof}

We can also observe that when $n$ is odd the distance between $0/1$ and $n/(n-1)$ grows very large in the even Farey graph but is constant in the normal Farey graph.  We can use the homeomorphism $L(n,n-1) \cong L(n,1)$ together with Proposition \ref{embedding_n_1} to see that $L(n,n-1)$ embeds almost in the spine of a minimal genus trisection of $S^2 \tilde \times S^2$.  However, if we only look for embeddings into connect sums of $S^2 \times S^2$s, not allowing twisted bundles, then we could only construct embeddings into $\#^{(n-1)} S^2 \times S^2$.  There may, of course, be other constructions that allow us to embed $L(n,n-1)$ almost into the spine of $\#^m S^2 \times S^2$ for smaller $m$.  However, our construction here provides some small evidence for the following conjecture:

\begin{conj}
Let $n_1(p,q)$ be the smallest integer for which $L(p,q)$ embeds almost in the spine of a minimal genus trisection of $\#^{n_1} S^2 \times S^2$ and $n_2(p,q)$ the smallest integer for which $L(p,q)$ embeds almost in the spine of a minimal genus trisection of $\#^{n_2} S^2 \tilde \times S^2$.  Then $|n_1 - n_2|$ is unbounded.
\end{conj}

If we go back to the more general case, we can also conjecture that any orientable 3--manifold can be embedded nicely into $\#^n S^2 \times S^2$ for large enough $n$.

\begin{conj}
If $M$ is an orientable 3--manifold, then $M$ embeds almost in the spine of a minimal genus trisection of $\#^n S^2 \times S^2$ for some positive integer $n$.
\end{conj}

This could possibly be proven with the same strategy we used for lens spaces.  First find a Heegaard diagram of $M$ such that when the Heegaard surface is embedded standardly into $S^3$, the induced framings on every curve of the diagram are all even.  Then find a path in $\mathcal{A}$ connecting a set of standard meridians with the diagram for $M$, such that every vertex of the path also represents a cut system with all even framings.  If this is done, the construction of Theorem \ref{generalembeddingthm} will give the desired embedding.

\section{embeddings in non minimal genus trisections}
Theorem \ref{generalembeddingthm} can be modified in a way that often results in a lower value of $n$ at the cost of the trisection of $\#^n S^2 \times S^2$ or $\#^n S^2 \tilde \times S^2$ not being minimal genus.  

\begin{thm}
Assume $M$ is a 3--manifold that is not $\#^k S^1 \times S^2$ with genus $g$ Heegaard surface $S$ (so in particular $d_{\mathcal{B}'}(S) > 1$).  Then $M$ lies almost in the spine of a trisection of $\#^n S^2 \times S^2$ or $\#^n S^2 \tilde \times S^2$ where $n = g(d_{\mathcal{B}'}(S)-1)$.
\end{thm}

Note that $d_{\mathcal{A}'}$ in Theorem \ref{generalembeddingthm} has been replaced with $d_{\mathcal{B}'}$ in this result.

\begin{proof}
Proceed identically to the original theorem until the analysis of the Kirby diagram where we have a sequence $\mathcal{D}_1,\cdots,\mathcal{D}_m,\mathcal{D}_{m-1},\cdots,\mathcal{D}_2$ lying on parallel copies of $S$, and with adjacent copies of $S$ tubed together to form the trisection surface giving the framings of all the curves in the Kirby diagram.  Previously the curves of the diagrams could be slid such that $\mathcal{D}_m$ and $\mathcal{D}_{m - 1}$ were dual (giving a Heegaard diagram of $S^3$, but now some of the curves may be parallel instead of dual.  When curves are dual, we follow the same process as we did before.  We slide the curve from the inner $\mathcal{D}_{m - 1}$ over the same curve from the outer $\mathcal{D}_{m - 1}$ to get a 0 framed unknot linking a curve of $\mathcal{D}_m$ once.  We then use the 0 framed unknot to separate the pair away from the rest of the diagram as a Hopf link representing a copy of $S^2 \times S^2$ or $S^2 \tilde \times S^2$.

If instead we have a curve of $\mathcal{D}_{m-1}$ parallel to a curve of $\mathcal{D}_m$, we take this curve from $\mathcal{D}_m$ and the inner $\mathcal{D}_{m-1}$ and slide both over the same curve from the outer $\mathcal{D}_{m-1}$ to get two 0 framed unknots unlinked from the rest of the diagram.  These 0 framed unknots can be removed without affecting the 4--manifold represented by the Kirby diagram.  In this way we can separate out every curve from the inner $\mathcal{D}_{m-1}$ and from $\mathcal{D}_m$ as before, resulting in the simpler Kirby diagram represented by the sequence $\mathcal{D}_1,\cdots,\mathcal{D}_{m-1},\mathcal{D}_{m-2},\cdots,\mathcal{D}_2$ peaking at $\mathcal{D}_{m-1}$.
\end{proof}

\section{embedding number and spine embedding number}
Fix an orientable 3--manifold $M$.  The \emph{embedding number} $\varepsilon(M)$ of $M$ is the smallest value of $n$ such that $M$ embeds in $\#^n S^2 \times S^2$.  It is known that there is always some such embedding, so the embedding number is always a finite integer.  This leads us to the following definitions:
\begin{defn}
The \emph{spine embedding number} $\varepsilon_S(M)$ of $M$ is the smallest integer value $n$ such that $M$ embeds almost in the spine of a trisection of $\#^n S^2 \times S^2$.  We take $n$ to be $\infty$ if no such embedding exists for any $n$.
\end{defn}
\begin{defn}
The \emph{twisted spine embedding number} $\tilde \varepsilon_S(M)$ of $M$ is the smallest integer value $n$ such that $M$ embeds almost in the spine of a trisection of $\#^n S^2 \tilde \times S^2$.   We have shown that $\tilde \varepsilon_S(M) < \infty$ for every $M$.
\end{defn}

By calculating the distance from $0/1$ to $p/q$ in the even Farey diagram, we can show that
\begin{prop}
$\varepsilon_S(L(p,q)) \le p-1$
\end{prop}
\begin{proof}
As when we originally found embeddings of $L(p,q)$ in the spines of connect sums of $\#^n S^2 \times S^2$s, we replace $q$ with $p-q$ if both $p$ and $q$ were odd. Construct the Farey diagram iteratively starting from the biggest triangles and adding smaller triangles in each iteration.  After the first iteration we have constructed the two largest triangles, and the set of vertices is $(1,1),(1,0),(0,1),(-1,1)$.  We focus on the region between $(1,0)$ and $(1,1)$.  In the second iteration, $2/1$ is added, which is distance $2$ from $0/1$ in both the Farey graph and even Farey graph.  In the third iteration, $3/1$ and $3/2$ are added.  Continuing in this manner, the $n$th iteration adds all fractions $n/m$ with $m < n$ that are in reduced form.  It is easy to calculate the distance from $(0,1)$ in the even Farey graph, since it is precisely the iteration in which the vertex was added.  Since $p/q$ will be added in the $p$th iteration at the latest, it is distance $\le p$ from $(0,1)$.  Using this path we can construct an embedding into $\#^n S^2 \times S^2$ with $n \le p-1$.
\end{proof}
Since $\varepsilon(M) \le \varepsilon_S(M)$, this leads to an alternative proof that $\varepsilon(L(p,q))$ is at most $p-1$ for any lens space, a fact proved previously in \cite{AcetoGollaLarson1}.  However, \cite{AcetoGollaLarson1} also gives constructions of cases where $\varepsilon(M)$ is less than the bound we find.  For example, it was also proved in  that $L(19,18))$ embeds in $\#^4 S^2 \times S^2$, but the distance in the even Farey graph only gives the bound $\varepsilon(L(19,18)) \le 18$.  It is unknown whether these embeddings can be realized as embeddings almost in the spine of a trisection.  However, it is unlikely that $\varepsilon(M)$ is always equal to $\varepsilon_S(M)$. 

\section{Further directions}

We conclude with some questions and conjectures.  The most obvious question is whether embeddings can always be realized as lying almost in the spine of a trisection.  More formally, we ask:

\begin{question}
If $M$ is a 3--manifold embedded in a smooth 4--manifold $X$, does there always exist a trisection of $X$ such that $M$ can be (topologically/piecewise linearly) isotoped to lie almost in the spine?
\end{question}

Recall that every 4--manifold can be embedded in a connect sum of $S^2 \times S^2$s, whereas our construction required using copies of $S^2 \tilde \times S^2$.  We ask if our trisection result can be improved to give a trisection theory version of the more standard result.

\begin{question}
Can every 3--manifold $M$ be embedded to lie almost in the spine of a trisection of some $\#^k S^2 \times S^2$?
\end{question}

We moreover ask
\begin{question}
Are the trisections we construct always standard?
\end{question}

If the answer is no, then our trisections may provide a large class of nonstandard trisections of connect sums of sphere bundles.  As of the writing of this paper, it has beens shown that there exist manifolds with non-isotopic trisections of the same genus \cite{Islambouli1}, but all such examples have nontrivial fundamental group.  It would therefore be interesting to find examples of non-isotopic trisections of the same genus for simply connected 4--manifolds.

An embedding almost in the spine of a trisection gives a decomposition of the 3--manifold into three collections of handlebodies (and some balls which can be collapsed).  We can observe that in all of the examples we have constructed these are trisections of the 3--manifold as per \cite{Koenig2}.  It would be interesting to know when a trisection of a 3--manifold can be realized by an embedding almost in the spine of a 4--manifold.

Lastly, it would be interesting to see nontrivial examples of 3--manifolds lying almost in the spine of a trisection of $S^4$.  Finding classes of such examples would give a different perspective as to why such 3--manifolds do embed into $S^4$.
\bibliographystyle{amsplain}
\bibliography{Trisections}{}

\providecommand{\bysame}{\leavevmode\hbox to3em{\hrulefill}\thinspace}
\providecommand{\MR}{\relax\ifhmode\unskip\space\fi MR }
\providecommand{\MRhref}[2]{%
  \href{http://www.ams.org/mathscinet-getitem?mr=#1}{#2}
}
\providecommand{\href}[2]{#2}
\begin{thebibliography}{10}

\bibitem{AcetoGollaLarson1}
Paolo Aceto, Marco Golla, and Kyle Larson, \emph{Embedding 3-manifolds in spin
  4-manifolds}, 2016.

\bibitem{GayKirby1}
David Gay and Robion Kirby, \emph{Trisecting 4-manifolds}, Geom. \& Topo.
  \textbf{20} (2016), no.~6, 3097--3132.

\bibitem{GompfStipsicz1}
Robert Gompf and Andr\'as Stipsicz, \emph{4--manifolds and {Kirby} calculus},
  Graduate studies in mathematics, vol.~20, American Mathematical Society,
  1999.

\bibitem{Islambouli1}
Gabriel Islambouli, \emph{Nielsen equivalence and trisections of 4-manifolds},
  arXiv:1804.06978, 2018.

\bibitem{Koenig2}
Dale Koenig, \emph{Triections of 3-manifolds}, arXiv:1805.11276, 2018.

\bibitem{LaudenbachPoenaru1}
Fran\c{c}ois Laudenbach and Valentin Po\'{e}naru, \emph{A note on 4-dimensional
  handlebodies}, Bull. Soc. Math. France \textbf{100} (1972), 337--344.

\bibitem{meier1}
Jeffrey Meier, \emph{Trisections and spun 4-manifolds}, 2017.

\bibitem{MeierZupan3}
Jeffrey Meier and Alexander Zupan, \emph{Bridge trisections of knotted surfaces
  in 4-manifolds}, arXiv:1710.01745, 2017.

\bibitem{MeierZupan2}
\bysame, \emph{Bridge trisections of knotted surfaces in {$S^4$}}, Trans. Amer.
  Math. Soc. \textbf{369} (2017), 7343--7386.

\bibitem{MeierZupan1}
\bysame, \emph{Genus two trisections are standard}, Geom. Topol. \textbf{21}
  (2017), 1583--1630.

\bibitem{Pao1}
Peter~Sie Pao, \emph{The topological structure of 4-manifolds with effective
  torus actions. i.}, Trans. Amer. Math. Soc. \textbf{227} (1977), 279--317.

\bibitem{HatcherThurston1}
William Thurston and Allen Hatcher, \emph{A presentation for the mapping class
  group of a closed orientable surface}, Topology \textbf{19} (1980), 221--237.

\bibitem{waldhausen1}
Friedhelm Waldhausen, \emph{Heegaard-zerlegungen der 3-sph{\"a}re}, Topology
  \textbf{7} (1968), 195--203.

\end{thebibliography}

%

%

\end{document}